\documentclass[11pt]{article}
\usepackage[margin=0.75in]{geometry}

\usepackage{amsthm}
\usepackage{amssymb}
\usepackage{amsmath}
\usepackage{comment}
\usepackage{thm-restate}
\usepackage{enumitem}
\usepackage{url} 
\usepackage{hyperref}
\usepackage[noabbrev,capitalise]{cleveref}

\usepackage[utf8]{inputenc}

\usepackage{color}
\usepackage[normalem]{ulem}

\newcommand{\Ag}{J}
\newcommand{\ec}{r}
\newcommand{\vc}{p}
\newcommand{\vn}{v}

\newcounter{i}
\theoremstyle{plain}
\newtheorem{thm}{Theorem}[section]

\newtheorem{claim}{Claim}[thm]

\newtheorem{question}[thm]{Question}

{\noindent \emph{Proof.} {}{#1}{}}{\hfill
	$\Diamond$\vspace{1em}}

\theoremstyle{plain} 
\newcommand{\thistheoremname}{}
\newtheorem{genericthm}{\thistheoremname}

\theoremstyle{definition}
\newtheorem{definition}[thm]{Definition}

\newcommand{\Prob}[1]{\ensuremath{%
\mathbb P\left[#1\right]
}}

\title{On generalized Ramsey numbers in the non-integral regime} 

\author{
Patrick Bennett
\thanks{Department of Mathematics, Western Michigan University, Kalamazoo, Michigan 49008, USA, {\tt Patrick.Bennett@wmich.edu}. Supported in part by Simons Foundation Grant No. 426894.}
\and
Michelle Delcourt
\thanks{Department of Mathematics, Toronto Metropolitan University,
Toronto, Ontario M5B 2K3, Canada, {\tt mdelcourt@torontomu.ca}. Research supported by NSERC under Discovery Grant No. 2019-04269.}
\and 
Lina Li
\thanks{Department of Mathematics,
Iowa State University, Ames, Iowa 50014, USA, {\tt linali@iastate.edu}.}
\and
Luke Postle
\thanks{Combinatorics and Optimization Department,
University of Waterloo, Waterloo, Ontario N2L 3G1, Canada, {\tt lpostle@uwaterloo.ca}. Partially supported by NSERC
under Discovery Grant No. 2019-04304.}}
\date{\today}

\begin{document}

\maketitle

\begin{abstract} 
A \emph{$(p,q)$-coloring} of a graph $G$ is an edge-coloring of $G$ such that every $p$-clique receives at least $q$ colors.  In 1975, Erd\H{o}s and Shelah introduced the \emph{generalized Ramsey number} $f(n,p,q)$ which is the minimum number of colors needed in a $(p,q)$-coloring of $K_n$.  In 1997, Erd\H{o}s and Gy\'arf\'as showed that
$f(n,p,q)$ is at most a constant times $n^{\frac{p-2}{\binom{p}{2} - q + 1}}$.
Very recently the first author, Dudek, and English improved this bound by a factor of $\log n^{\frac{-1}{\binom{p}{2} - q + 1}} $ for all $q \le \frac{p^2 - 26p + 55}{4}$, and they ask if this improvement could hold for a wider range of $q$.

We answer this in the affirmative for the entire non-integral regime, that is, for all integers $p, q$ with $p-2$ not divisible by $\binom{p}{2} - q + 1$.
Furthermore, we provide a simultaneous three-way generalization as follows: where $p$-clique is replaced by any fixed graph $F$ (with $|V(F)|-2$ not divisible by $|E(F)| - q + 1$); to list coloring; and to $k$-uniform hypergraphs. Our results are a new application of the Forbidden Submatching Method of the second and fourth authors.
\end{abstract}

\section{Introduction}

In 1975, Erd\H{o}s and Shelah~\cite{E75} proposed a generalization of Ramsey numbers as follows. A \emph{$(p,q)$-coloring} of a graph $G$ is an edge-coloring of $G$ such that every $p$-clique receives at least $q$ colors. The \emph{generalized Ramsey number} $f(n,p,q)$ is the minimum number of colors in a $(p,q)$-coloring of $K_n$. This notion was further developed by Erd\H{o}s and Gy\'arf\'as~\cite{EG97} in 1997 and has attracted much attention over the recent decades (e.g.~\cite{AFM00, BEHK21, BCDP22, CFLS15, JM22, M98}). We refer the reader to the survey chapter on hypergraph Ramsey problems by Mubayi and Suk~\cite{MS20} as well as the recent paper of the first author, Dudek and English~\cite{BDE22} for more history of the area.

In 1997, Erd\H{o}s and Gy\'arf\'as~\cite{EG97} proved the following via a simple application of the Lov\'asz Local Lemma.
\begin{thm}[Erd\H{o}s and Gy\'arf\'as~\cite{EG97}]\label{thm:EG}
For fixed positive integers $p,q$ with $p>2$ and $1\leq q \leq \binom{q}{2}$, we have
$$f(n,p,q) = O\left( n^{\frac{p-2}{\binom{p}{2} - q + 1}} \right).$$
\end{thm}

Furthermore, Erd\H{o}s and Gy\'arf\'as~\cite{EG97} determined the linear and quadratic thresholds for this problem as follows: more specifically, if $q=\binom{p}{2} - p +3$, they showed that $f(n,p,q) = \Theta(n)$, while Theorem~\ref{thm:EG} implies that if $q=\binom{p}{2}-p+2$, then $f(n,p,q)=o(n)$; similarly if $q = \binom{p}{2} - \lfloor \frac{p}{2} \rfloor + 2$, they showed that $f(n,p,q) = \Omega(n^2)$ while if $q=\binom{p}{2} - \lfloor \frac{p}{2} \rfloor + 1$, then Theorem~\ref{thm:EG} implies that $f(n,p,q)=o(n^2)$.

Very recently, the first author, Dudek and English~\cite{BDE22} improved the bound in Theorem~\ref{thm:EG} by a logarithmic factor for a wide range of values of $p$ and $q$ in the sublinear regime as follows.

\begin{thm}[Bennett, Dudek and English~\cite{BDE22}]\label{thm:BDE}
For fixed positive integers $p,q$ with $q \le \frac{p^2 - 26p + 55}{4}$, we have
$$f(n,p,q) = O\left( \left(\frac{n^{p-2}}{\log n}\right)^{\frac{1}{\binom{p}{2} - q + 1}} 
\right).$$
\end{thm}

Note this is roughly the range up to $q\approx \frac{p^2}{4}$. Their proof is quite involved and proceeds via a random greedy process. The authors asked if the bounds could be improved even up to the $q\approx \frac{p^2}{2}$ range.

We answer this in the affirmative for all $q$ below the linear threshold and furthermore for all $q$ where the exponent is non-integral with our first main result as follows.

\begin{thm}\label{thm:main1}
For fixed positive integers $p,q$ with $p-2$ not divisible by $\binom{p}{2} - q + 1$, we have
$$f(n,p,q) = O\left( \left(\frac{n^{p-2}}{\log n}\right)^{\frac{1}{\binom{p}{2} - q + 1}} 
\right).$$
\end{thm}

As indicated by the linear and quadratic thresholds discussed earlier, in the integral regime, the Erd\H{o}s-Gy\'arf\'as bound (Theorem~\ref{thm:EG}) is tight (up to a constant factor). This implies that Theorem~\ref{thm:main1} is the best possible in terms of the range of $p$ and $q$. Our theorem also provides the \textit{first improvement} on the general upper bound of $f(n, p, q)$ for all $q\geq p$, i.e., in the linear and superlinear regimes, since the local lemma bound given by Erd\H{o}s-Gy\'arf\'as in 1997.

From our result, it is now evident that for the non-integral regime, the local lemma upper bound is not close to the truth. A natural thought is to seek lower bounds. 
However, unlike upper bounds, there is no general formula known for lower bounds on all pairs of $p$ and $q$. In fact, the situation is more complicated for lower bounds, and there are many values of $(p, q)$ for which very little is known about lower bounds of $f(n, p, q)$. For pairs of $(p, q)$ with known lower bounds, we refer interested readers to the recent papers \cite{BEHK21, fish2020local, pohoata2019local}.

It is worth noting that our proof is independent of the proof of Theorem~\ref{thm:BDE} in~\cite{BDE22}, which spans 49 pages. In contrast, our actual proof takes less than 3 pages, following from a novel application of the recently developed and powerful method — the \emph{Forbidden Submatching Method}.

The Forbidden Submatching Method is a new technique in combinatorics for finding perfect matchings in hypergraphs while avoiding certain forbidden submatchings, that was recently introduced by the second and fourth author~\cite{DP22} (and also independently introduced as \emph{conflict-free hypergraph matchings} by Glock, Joos, Kim, K\"uhn and Lichev~\cite{GJKKL22}). The latter proceeds via a random greedy process while the former uses the nibble method; we discuss the differences between the two formulations and limitations of the latter in the conclusion. Although the method has only recently been introduced, it has already been applied to resolve a number of fundamental conjectures from many different areas such as Steiner systems, Latin squares, high dimensional permutations and degenerate Tur\'an densities. In~\cite{DP22} and~\cite{GJKKL22}, it was used to prove the existence of approximate high girth Steiner systems answering a conjecture of Glock, K\"uhn, Lo and Osthus~\cite{GKLO20}. In~\cite{DP22BES} and~\cite{GJKKLP22}, it was used to prove a conjecture of Brown, Erd\H{o}s and S\'os~\cite{BES73}. In~\cite{JM22}, it was applied to $(p,q)$-coloring to provide a short proof that $f(n,4,5)=\frac{5n}{6}+o(n)$.

The main contribution of this paper is to demonstrate that the machinery of the Forbidden Submatching Method can be applied to $(p,q)$-coloring in the non-integral regime as well as to various generalizations which we will discuss in the next subsection. Whereas a number of the applications have been mostly immediate applications of the method, we note that Theorem~\ref{thm:main1} is not a straightforward application of the method as the natural ways of encoding the problem into the setting of hypergraph matchings do not meet the required conditions. Indeed, a main innovation of this paper is the introduction of a \emph{potential function} on submatchings to define the right set of forbidden submatchings that do meet the needed conditions. Altogether, we believe that the Forbidden Submatching Method has much promise for future applications.

\subsection{Generalizations and Applications}

We further generalize Theorem~\ref{thm:main1} in three ways simultaneously as follows.

Let $F$ be a graph. An \emph{$(F,q)$-coloring} of a graph $G$ is an edge-coloring of $G$ such that every copy of $F$ in $G$ receives at least $q$ colors. The \emph{generalized Ramsey number} $r(G,F,q)$ is the minimum number of colors in an $(F,q)$-coloring of $G$.

Axenovich, F{\"u}redi, and Mubayi~\cite{AFM00} generalized Theorem~\ref{thm:EG} also via a simple application of the Lov\'asz Local Lemma as follows.
\begin{thm}[Axenovich, F{\"u}redi, and Mubayi~\cite{AFM00}]\label{thm:AFM}
Let $F$ be a fixed graph with $|V(F)|>2$. If $G$ is graph on $n$ vertices  and $q$ is a positive integer with $q \le |E(F)|$, then
$$r(G,F,q) = O\left( n^{\frac{|V(F)|-2}{|E(F)| - q + 1}} \right).$$
\end{thm}
Note that $r(G,F,q)$ is sublinear if $q < |E(F)|-|V(F)|+3$; moreover, when $G=K_n$ and $F$ is connected, Axenovich, F{\"u}redi, and Mubayi further showed that $r(K_n, F, q)=\Theta(n)$ for $q=|E(F)|-|V(F)|+3$. 

So our first generalization is to replace $p$ (more specifically the implied $K_p$) in Theorem~\ref{thm:main1} by a general $F$ with $|V(F)|-2$ not divisible by $|E(F)| - q + 1$. 

Our second generalization is to list coloring. A \emph{$k$-list-assignment} $L$ of the edges of a graph $G$ is a collection of lists $(L(e):e\in E(G))$ such that $|L(e)|\ge k$ for all $e\in E(G)$. The \emph{generalized list Ramsey number} $r_{\ell}(G,F,q)$ is the minimum $k$ such that for every $k$-list-assignment $L$ of $E(G)$ there exists an $(F,q)$-coloring $\phi$ such that $\phi(e)\in L(e)$ for all $e\in E(G)$.  Note that $r(G,F,q)\le r_{\ell}(G,F,q)$ for all $G,F,q$. 

Our third generalization is to hypergraphs.
The definition of $r(G,F,q)$ and $ r_{\ell}(G,F,q)$ naturally generalizes to hypergraphs $G$ and $F$; indeed, Erd\H{o}s and Shelah's original question concerned the generalized Ramsey number for uniform hypergraphs. We note that the Lov\'asz Local Lemma (e.g. see~\cite{AS16}) easily yields the following generalization of Theorem~\ref{thm:AFM} to list coloring and to $k$-uniform hypergraphs.

\begin{thm}\label{thm:base}
Let $k\ge 2$ be a positive integer. Let $F$ be a fixed $k$-uniform hypergraph with $|V(F)|>k$ and let $q$ be a positive integer with $q\leq |E(F)|$. If $G$ is $k$-uniform hypergraph on $n$ vertices, then
$$r_{\ell}(G,F,q) = O\left( n^{\frac{|V(F)|-k}{|E(F)| - q + 1}} \right).$$
\end{thm}
\begin{proof}
When $q=1$, the result is trivial, so we assume that $q\geq 2$. Let $t=C\left( n^{\frac{|V(F)|-k}{|E(F)| - q + 1}} \right)$ (where $C$ is a large constant to be chosen later), and without loss of generality, let $L$ be an arbitrary $t$-list-assignment of $E(G)$ with $|L(e)|=\ell$.
Color each edge $e\in G$ independently with colors from its own list $L(e)$, where the colors are assigned with equal probability, i.e., $1/\ell$. For each copy $T$ of $F$ in $G$, let $\mathcal{A}_T$ be the event that $T$ receives at most $q-1$ colors; clearly
\[
\Prob{\mathcal{A}_T} \le p:=\binom{|E(F)|\cdot t}{q-1}\left(\frac{q-1}{t}\right)^{|E(F)|}\sim t^{-|E(F)| + q-1}.
\]
Note that two events $\mathcal{A}_T$ and $\mathcal{A}_{T'}$ are dependent only if $E(T)\cap E(T')\neq \emptyset$. Therefore each event $\mathcal{A}_T$ is mutually
independent of a set of all the other events but at most 
\[
d\le |E(F)|\cdot \binom{n}{|V(F)|-k} \sim n^{|V(F)|-k}.
\]
Choosing $C$ sufficiently large we obtain $ep(d+1) \le 1$. It follows from the symmetric case of the Local Lemma that with positive probability there exists a $(F,q)$-coloring, thereby completing the proof.
\end{proof}

Here then is our general main result which improves the bound in Theorem~\ref{thm:base} by a logarithmic factor in the non-integral regime. 

\begin{thm}\label{thm:main2}
Let $k\ge 2$ be a positive integer. Let $F$ be a fixed $k$-uniform hypergraph with $|V(F)|>k$ and let $q$ be a positive integer with $q\leq |E(F)|$ and $|V (F)|-k$ not divisible by $|E(F)|-q +1$. If $G$ is a $k$-uniform hypergraph on $n$ vertices, then 
$$ r_{\ell}(G,F,q) = O\left( \left(\frac{n^{|V(F)|-k}}{\log n}\right)^{\frac{1}{|E(F)| - q + 1}} \right)
.$$
\end{thm}

As the first author, Dudek, and English noted in~\cite{BDE22}, upper bounds on generalized Ramsey numbers have implications for the lower bounds of a problem of Brown, Erd\H{o}s and S\'os~\cite{BES73} as follows. Let $F^{(k)}(n; j, i)$ denote the minimum number of edges $m$ such that every $k$-uniform hypergraph on $n$ vertices and $m$ edges has a set of $j$ vertices spanning at least $i$ edges.  In this notation, the second and fourth authors~\cite{DP22BES} recently showed the following, confirming a conjecture of Brown, Erd\H{o}s and S\'os~\cite{BES73} for 3-uniform hypergraphs.
\begin{thm}[Delcourt and Postle~\cite{DP22BES}]
For all $j \geq 4$, 
$$\lim_{n \rightarrow \infty} n^{-2}F^{(3)}(n; j, j-2)$$ exists.
\end{thm}

For general $i,j,k$, Brown, Erd\H{o}s, and S\'os~\cite{BES73} showed that $F^{(k)}(n; j, i) = \Omega\left(n^{\frac{k\cdot i - j}{i-1}}\right)$.  For graphs, the first author, Dudek, and English~\cite{BDE22} observed that $F^{(2)}(n; j, i) \geq \frac{{n \choose 2}}{f\left(n,~j,~{j \choose 2} -i+2\right)}$, since a $\left(j, {j\choose 2}-i+2\right)$-coloring of $K_n$ contains no set of $j$ vertices which spans $i$ edges of the same color (and consequently every color class has at most $F^{(2)}(n; j, i)$ edges).
This together with Theorem~\ref{thm:BDE} improves the lower bound of $F^{(2)}(n; j, i)$ to $\Omega\left(\left(n^{2i - j} \log n\right)^{\frac{1}{i-1}}\right)$ for all $i \geq \frac{j^2+24j-47}{4}$.

Similarly for $k$-uniform hypergraphs, one can observe that for general $i, j$,
$$F^{(k)}(n; j, i) \geq \frac{{n \choose k}}{r\left(K_n^{(k)}, K_j^{(k)}, {j \choose k} -i+2\right)}.$$  Hence, Theorem~\ref{thm:main2} implies the following.

\begin{thm}\label{thm:bes}
For all fixed $i,j,k$ with $ki-j$ not divisible by $i-1$, 
$$F^{(k)}(n; j, i) = \Omega\left(\left(n^{k\cdot i - j} \log n\right)^{\frac{1}{i-1}}\right).$$
\end{thm}

We note that Theorem~\ref{thm:bes} improves an earlier result of Shangguan and Tamo~\cite{shangguan2020sparse}, where they proved the same result but under the assumption that $ki-j$ and $i-1$ are relatively prime. Moreover, Theorem~\ref{thm:bes} is the best possible in terms of the range of $k,j,i$, as it was shown by~\cite{BES73} that $F^{(k)}(n; j, i)=\Theta\left(n^{(ki-j)/(i-1)}\right)$ for $i-1 \mid ki-j$.

\section{Proof of Theorem~\ref{thm:main2}}

As mentioned in the introduction, Theorem~\ref{thm:main2} will follow as an application of the main theorem of~\cite{DP22} by the second and fourth authors where they developed a very general theory of finding matchings in hypergraphs avoiding forbidden submatchings. In order to state that theorem, we first need a few definitions. 

\begin{definition}[Configuration Hypergraph]
Let $\Ag$ be a (multi)-hypergraph. We say a hypergraph $H$ is a \emph{configuration hypergraph} for $\Ag$ if $V(H)=E(\Ag)$ and $E(H)$ consists of a set of matchings of $\Ag$ of size at least two. We say a matching of $\Ag$ is \emph{$H$-avoiding} if it spans no edge of $H$.
\end{definition}

\begin{definition}[Bipartite Hypergraph]
We say a hypergraph $J=(A,B)$ is \emph{bipartite with parts $A$ and $B$} if $V(J)=A\cup B$ and  every edge of $J$ contains exactly one vertex from $A$. We say a matching of $J$ is \emph{$A$-perfect} if every vertex of $A$ is in an edge of the matching.
\end{definition}

\begin{definition}
Let $H$ be a hypergraph. The \emph{$i$-degree} of a vertex $v$ of $H$, denoted $d_{H,i}(v)$, is the number of edges of $H$ of size $i$ containing $v$. The \emph{maximum $i$-degree} of $H$, denoted $\Delta_i(H)$, is the maximum of $d_{H,i}(v)$ over all vertices $v$ of $H$.

The \emph{$s$-codegree} of a set of vertices $S$ of $H$, denoted $d_{H,s}(S)$ is the number of edges of $H$ of size $s$ containing $S$. The \emph{maximum $(s,t)$-codegree} of $H$ is $$\Delta_{s,t}(H) := \max_{S\in \binom{V(H)}{t}} d_{H,s}(S).$$
We define the \emph{common $2$-degree} of distinct vertices $u, v\in V(H)$ as $|\{w\in V(H): uw, vw\in E(H)\}|$. Similarly, we define the \emph{maximum common $2$-degree} of $H$ as the maximum of the common $2$-degree of $u$ and $v$ over all distinct pairs of vertices $u,v$ of $H$.
\end{definition}

\begin{definition}
Let $\Ag$ be a hypergraph and let $H$ be a configuration hypergraph of $\Ag$. We define the \emph{$i$-codegree} of a vertex $v\in V(\Ag)$ and $e\in E(\Ag)=V(H)$ with $v\notin e$ as the number of edges of $H$ of size $i$ who contain $e$ and an edge incident with $v$. We then define the \emph{maximum $i$-codegree} of $\Ag$ with $H$ as the maximum $i$-codegree over vertices $v\in V(\Ag)$ and edges $e\in E(\Ag)=V(H)$ with $v\notin e$. 
\end{definition}

We are now ready to state the main theorem from~\cite{DP22}.

\begin{thm}\label{thm:SmallCodegree}[Theorem 1.16 from~\cite{DP22}]
For all integers $r_1,r_2 \ge 2$ and real $\beta \in (0,1)$, there exist an integer $D_{\beta}\ge 0$ and real $\alpha > 0$ such that following holds for all $D\ge D_{\beta}$: 
\vskip.05in
Let $\Ag=(A,B)$ be a bipartite $r_1$-bounded (multi)-hypergraph with codegrees at most $D^{1-\beta}$ such that every vertex in $A$ has degree at least $(1+D^{-\alpha})D$ and every vertex in $B$ has degree at most $D$. Let $H$ be an $r_2$-bounded configuration hypergraph of $\Ag$ with $\Delta_i(H) \le \alpha \cdot D^{i-1}\log D$ for all $2\le i\le r_2$ and $\Delta_{s,t}(H) \le D^{s-t-\beta}$ for all $2\le t< s\le r_2$. If the maximum $2$-codegree of $\Ag$ with $H$ and the maximum common $2$-degree of $H$ are both at most $D^{1-\beta}$, then there exists an $H$-avoiding $A$-perfect matching of $\Ag$ and indeed even a set of $D$ disjoint $H$-avoiding $A$-perfect matchings of $\Ag$.
\end{thm}

Theorem~\ref{thm:SmallCodegree} is very general:~namely, it simultaneously generalizes the famous matching theorem of Pippenger and Spencer~\cite{PS89} (in the case when $H$ is empty albeit with weaker codegree assumptions) and (vertex) independent set results for girth five hypergraphs or hypergraphs with small codegrees, for example the celebrated result of Ajtai, Koml{\'o}s, Pintz, Spencer, and Szemer{\'e}di~\cite{AKPSS82} (in the case when $J$ is a matching). 

Indeed, the bipartite hypergraph $A$-perfect matching formulation of Theorem~\ref{thm:SmallCodegree} is stronger than all of the above as it even implies the coloring and list coloring versions of those results. Moreover, we will need the full strength of this formulation to prove Theorem~\ref{thm:main2} as will become clear from the proof.

To prove Theorem~\ref{thm:main2}, we will apply Theorem~\ref{thm:SmallCodegree} to certain well-chosen $J$ and $H$ to find an $H$-avoiding $A$-perfect matching of $J$ which will correspond to an $(F,q)$-coloring of $G$. The main idea is to set $A$ to be the edges of $G$ and define the edges of $J$ to correspond to edge-color pairs $(e,c)$ where $c\in L(e)$ in such a way that $J$ is a graph that is the disjoint union of stars. Thus an $A$-perfect matching of $J$ will correspond to an $L$-coloring of the edges of $G$.

To enforce that the coloring is an $(F,q)$-coloring we would naturally define our edges of $H$ to be the sets of edge-color pairs which would correspond to a coloring of a copy of $F$ with at most $q-1$ colors. By setting these to be the forbidden configurations, it follows that an $H$-avoiding $A$-perfect matching of $J$ corresponds to an $(F,q)$-coloring of $G$.

It turns out that the gained logarithmic factor in Theorem~\ref{thm:main2} precisely results from the $\log D$ factor in Theorem~\ref{thm:SmallCodegree}. Unfortunately, this natural definition of $H$ does not satisfy the small codegree conditions of Theorem~\ref{thm:SmallCodegree}, namely that $\Delta_{s,t}(H) \le D^{s-t-\beta}$. Indeed, there can be subconfigurations of the bad configurations whose codegrees are too large. 

Thus the main innovation in our proof of Theorem~\ref{thm:main2} is to define a \emph{potential function} over all possible subconfigurations and then define $H$ to be the set of minimal subconfigurations with potential at least that of the $(q-1)$-colored copies of $F$. The small codegree condition will then follow by the minimality of $H$. Importantly, we are thus using that $H$ in Theorem~\ref{thm:SmallCodegree} is allowed to have mixed uniformities. Intriguingly, the only place we use that $|V(F)|-k$ is not divisible by $|E(F)|-q+1$ is to show that the maximum $2$-degree of $H$ is small (and so are the maximum $2$-codegree of $J$ with $H$ and the common $2$-degree of $H$).

We are now prepared to prove Theorem~\ref{thm:main2} as follows.

\begin{proof}[Proof of Theorem~\ref{thm:main2}]
For ease of notation, let $\vc:= |V(F)|$ and $\ec:= |E(F)|$. Set $T:= \left\lceil C\cdot \left(\frac{n^{\vc-k}}{\log n}\right)^{\frac{1}{\ec - q + 1}} \right\rceil
$. Note we will choose $n, C$ sufficiently large to satisfy various inequalities that appear throughout the proof. Let $L$ be a $T$-list-assignment of $E(G)$. We assume without loss of generality that $|L(e)|=T$ for all $e\in E(G)$.

We define a bipartite graph $J=(A,B)$ as follows:
\begin{itemize}\itemsep.05in
\item $A:=E(G)$
\item $B:= \{(e,c): e\in E(G),~c\in L(e)\}$
\item $E(J) := \big\{ \{e,(e,c)\}: e\in E(G),~c\in L(e)\big\}.$
\end{itemize}
For convenience of notation, we will refer to the edge $\{e,(e,c)\}$ as simply $(e,c)$.

Let $\beta := \frac{1}{2(p-k)}$, $r_1:=2$ and $r_2:=\ec$. Recall that $p>k$ by assumption and hence $\beta > 0$. Let $D_{\beta}$ and $\alpha$ be as in Theorem~\ref{thm:SmallCodegree} for $r_1,r_2$ and $\beta$. Let $D:= \frac{T}{2}$.  Since $n$ and $C$ are large enough, we have that $D\ge D_{\beta}$. 

Note that $J$ is a bipartite $2$-bounded hypergraph with codegrees at most $1 \le D^{1-\beta}$ such that every vertex in $A$ has degree at least $T = 2D \ge (1+D^{-\alpha})D$ (since $\alpha > 0$) and every vertex in $B$ has degree at most $1\le D$.

For a matching $M$ of $J$, denote by $E(M)$ the set of edges $e$ of $G$ such that there exists $f\in M$ with $e\in f$. Note that $|M|=|E(M)|$ since $M$ is a matching.
We let $V(M)$ denote the set of vertices $v$ of $G$ spanned by $E(M)$, and $C(M)$ denote the set of colors $c$ such that there exists $f=(e,c)\in M$ for some edge $e$ of $G$.
We then define the \emph{potential} of a matching $M$ of $J$ as
$$\rho(M):= |M| - |C(M)| - \left( |V(M)|-k\right)\cdot \left(\frac{r-q+1}{p-k} \right).$$

Let $\kappa_F$ denote the set of subgraphs of $G$ that are isomorphic to $F$.
The configuration hypergraph $H$ of $J$ is defined as follows. We let $V(H)=E(J)$ and $E(H)$ be the set of matchings $M$ of $J$ such that
\begin{itemize}\itemsep.05in
    \item $|M|\ge 2$,
    \item $E(M) \subseteq E(K)$ for some $K\in \kappa_F$,
    \item $\rho(M)\ge 0$, 
    \item and subject to those conditions $M$ is inclusion-wise minimal.
\end{itemize}
Note that $H$ is $\ec$-bounded since for every $K\in \kappa_F$, we have $|E(K)|\le \ec$. Furthermore, if $M$ is a matching of $J$ with $|M|=\ec$, $|V(M)|=p$ and $|C(M)|\le q-1$, then \begin{equation}\label{posipo}
    \rho(M) \ge \ec - (q-1) - (p-k) \cdot \left(\frac{\ec-q+1}{p-k} \right) = 0.
\end{equation}

Suppose $Z$ is an $H$-avoiding $A$-perfect matching of $J$. Let $\phi$ be an edge coloring of $G$ defined as $\phi(e) := c$ where $(e,c)\in Z$. Since $Z$ is an $A$-perfect matching of $J$, we have that $\phi$ is well-defined for all edges $e$ of $G$. Since $\left(e,\phi(e)\right)\in E(J)$, it follows that $\phi(e) \in L(e)$ for every edge $e$ of $G$. Furthermore, since $Z$ is $H$-avoiding, it follows from \eqref{posipo} that $\phi$ is an $(F,q)$-coloring of $G$ (since otherwise it would contain a matching $M$ with $|M|=r$, $|V(M)|=p$ and $|C(M)|\le q-1$ and hence would contain an edge of $H$ contradicting that $Z$ is $H$-avoiding). 

Thus it suffices to prove that there exists an $H$-avoiding $A$-perfect matching of $J$. To that end, we verify that $J$ and $H$ satisfy the hypotheses of Theorem~\ref{thm:SmallCodegree} where $\beta, r_1,r_2$ and $D$ are defined as above. 

Let $\mathcal{I}$ denote the set of triples of integers $(m,\vn, \ell)$ with $2\le m \le \ec,~k+1\le \vn\le p,~1\le \ell \le \ec$ where 
\begin{equation}\label{rela:mnl}
m \ge \ell + (\vn-k)\cdot \left(\frac{\ec-q+1}{\vc-k} \right).
\end{equation}
Note that $|\mathcal{I}|$ is a constant that depends only on $k,~p,~q$ and $\ec$.
For a triple $(m,\vn, \ell)\in \mathcal{I}$, we let $H_{m,\vn,\ell}$ be the $m$-uniform configuration hypergraph of $J$ where $E(H_{m,\vn,\ell})$ is the set of edges $M$ of $H$ with $|M|=m$, $|V(M)|=\vn$ and $|C(M)|=\ell$. Note that by definition of $H$, we have $H=\bigcup_{(m,v,\ell)\in \mathcal{I}} H_{m,v,\ell}$.

We first verify that the maximum 2-codegree of $J$ with $H$ and the maximum common 2-degree of $H$ are as desired in Theorem~\ref{thm:SmallCodegree}.
Observe that both of them are at most $\Delta_2(H)$. Therefore, it is sufficient to show the following claim.
\begin{claim}
$\Delta_2(H)\leq D^{1-\beta}$.
\end{claim}
\begin{proof}
Let $M=\{f, f'\}$ be a 2-configuration in $H$ with $f=(e, (e, c))$ and $f'=(e', (e', c'))$. By the definition of $H$, we have 
\[
\rho(M)=2-|C(M)|-\left(|V(M)|-k\right)\cdot \left(\frac{r-q+1}{p-k} \right)\geq 0.
\]
Since $|V(M)|>k$ and $q\leq r$, we must have $|C(M)|\leq 1$, which gives $c=c'$. Furthermore, this indicates that
\[
|V(M)|\leq k + \left\lfloor\frac{p-k}{r-q+1} \right\rfloor\leq k + \frac{p-k-1}{r-q+1},
\]
where the last inequality\footnote{Indeed, this is the only place throughout the proof which uses the non-divisibility.} follows since $p-k$ is not divisible by $r-q+1$.

By the definition of $D$, there exists a constant $C_0$ such that
$
n\leq C_0(\log D)^{\frac{1}{p-k}}\cdot D^{\frac{r-q+1}{p-k}}.
$
Therefore by the above discussion, for any given $f$, the number of 2-configurations containing $f$ is at most
\[
n^{\frac{p-k-1}{r-q+1}}\leq C_0^{\frac{p-k-1}{r-q+1}}\log D\cdot D^{1 -\frac{1}{p-k}}\leq D^{1-\beta},
\]
which completes the proof.
\end{proof}

Next, we check the degrees of $H$. 
Let $(e, c)\in V(H)$, and assume $e=\{u_1,u_2,\ldots,u_k\}$.
Fix an arbitrary triple $(m,\vn,\ell)\in\mathcal{I}$.
In particular, by the choices of $\vn, k, p, q, r$ and \eqref{rela:mnl}, we have 
\begin{equation}\label{ineq:ml}
m-\ell> 1.
\end{equation}
We now calculate an upper bound on the number of $M\in E(H_{m,\vn,\ell})$ such that $(e, c)\in M$ as follows. First, there are at most $n^{\vn-k}$ choices for the $\vn-k$ vertices of $V(M)\setminus \{u_1,u_2,\ldots,u_k\}$.
Second, there are at most $T^{\ell-1}$ choices for the $\ell-1$ colors in $C(M)\setminus \{c\}$.
Once we have fixed $V(M)$ and $C(M)$, the number of choices for $M$ is at most $O(1)$.
Therefore, there exists a constant $C_1$ such that
\begin{align*}
\frac{d_{H_{m,\vn,\ell}}((e, c))}{D^{m-1}\log D}
&\leq C_1\cdot \frac{n^{\vn-k}\cdot T^{\ell-1}}{T^{m-1}\log n} \leq C_1\cdot C^{\ell-m}\cdot \frac{n^{\vn-k}}{\log n}\left(\frac{n^{\vc-k}}{\log n}\right)^{\frac{\ell-m}{\ec - q + 1}}.
\end{align*}
Since $n$ and $C$ are chosen to be sufficiently large and since $p > k$, we have $C,\frac{n^{\vc-k}}{\log n}>1$. Therefore by \eqref{rela:mnl}, \eqref{ineq:ml}, and the monotonicity of the function $a^x$ (where $a>1$), we have
\begin{align*}
\frac{d_{H_{m,\vn,\ell}}((e, c))}{D^{m-1}\log D}
& \leq \frac{C_1}{C}\cdot \frac{n^{\vn-k}}{\log n}\left(\frac{n^{\vc-k}}{\log n}\right)^{-\frac{\vn-k}{\vc-k}} = \frac{C_1}{C} \cdot (\log n)^{\frac{\vn-k}{\vc-k}-1}\leq \frac{\alpha}{|\mathcal{I}|},
\end{align*}
where the last inequality follows since $C$ is chosen to be large enough and $\vn\leq \vc$. Therefore, for any $2\leq i\leq \ec$ and $(e,c)\in V(H)$, we have
\[
d_{H, i}((e, c))=\sum_{(i,\vn,\ell)\in\mathcal{I}}d_{H_{i,\vn,\ell}}((e, c))\leq \alpha \cdot D^{i-1}\log D.
\]

We now check the codegrees of $H$. Fix integers $s,t$ with $2\le t < s \le r_2$.  Let $S$ be a subset of $V(H)$ of size $t$. We desire to show that $d_{H,s}(S) \le D^{s-t-\beta}$. If $S$ is not strictly contained in at least one edge of $H$, then $d_{H,s}(S)=0$ as desired. So we assume that $S$ is strictly contained in some edge $f$ of $H$. 
By definition of $H$,  we have that $E(f) \subseteq E(K)$ for some $K\in \kappa_F$ and $\rho(f)\ge 0$. Since $|S|=t\ge 2$ and $S\subsetneq f$, we have by the minimality of $H$ that $\rho(S)<0$.

Recall that $|S|=t$, and let $\vn' := |V(S)|$ and $\ell':=|C(S)|$. Since $\rho(S)\cdot (p-k)$ is an integer, it follows that $\rho(S) \le -\frac{1}{p-k}$ and hence
\begin{equation}\label{rela:mnl'}
t \le \ell' + (\vn'-k)\left(\frac{\ec-q+1}{\vc-k} \right)-\frac{1}{\vc-k}.
\end{equation}
Fix an arbitrary triple $(s,\vn,\ell)\in\mathcal{I}$ with $s>t,~\vn\ge \vn'$ and $\ell\ge \ell'$. 
Let $\mu:=(\ell-s)-(\ell'-t)$, and recall that $\beta=\frac{1}{2(\vc-k)}$. Then by \eqref{rela:mnl} and \eqref{rela:mnl'}, we have 
\begin{equation}\label{ineq:mu}
   \mu\leq -(\vn-\vn')\left(\frac{\ec-q+1}{\vc-k} \right)-2\beta, 
\end{equation}
and in particular, $\mu+\beta <0$.

We now calculate an upper bound on the number of $M\in E(H_{s,\vn,\ell})$ such that $S\subseteq M$ as follows. Similarly as before, there are at most $n^{\vn-\vn'}$ choices for the $\vn-\vn'$ vertices of $V(M)\setminus V(S)$, and at most $T^{\ell-\ell'}$ choices for the $\ell-\ell'$ colors in $C(M)\setminus C(S)$. Once we have fixed $V(M)$ and $C(M)$, the number of choices for $M$ is at most $O(1)$. Therefore, there exists a constant $C_2$ such that
\begin{align*}
\frac{d_{H_{s,\vn,\ell}}(S)}{D^{s-t-\beta}}
&\leq C_2\cdot \frac{n^{\vn-\vn'}\cdot T^{\ell-\ell'}}{T^{s-t-\beta}} 
\leq C_2\cdot C^{\mu+\beta}\cdot n^{\vn-\vn'}\left(\frac{n^{\vc-k}}{\log n}\right)^{\frac{\mu+\beta}{\ec - q + 1}}.
\end{align*}
Once again, since $n$ and $C$ are chosen to be sufficiently large and $p>k$, we have $C,\frac{n^{\vc-k}}{\log n}>1$. Therefore by \eqref{ineq:mu} and the monotonicity of the function $a^x$ (where $a>1$), we have
\begin{align*}
\frac{d_{H_{s,\vn,\ell}}(S)}{D^{s-t-\beta}}
& \leq C_2\cdot n^{\vn-\vn'}\left(n^{\vc-k}\right)^{-\frac{\vn-\vn'}{\vc-k}} \left(n^{\vc-k}\right)^{-\frac{\beta}{r-q+1}} 
= C_2\cdot n^{-\frac{\beta(\vc-k)}{r-q+1}} \leq \frac{1}{|\mathcal{I}|},
\end{align*}
where the last inequality follows since $n$ is chosen to be large enough and $\vc > k$.

Hence, for any $2\leq t< s\leq r_2$, we have
\[
\Delta_{s,t}(H) = \max_{S\in \binom{V(H)}{t}} d_{H,s}(S)
\leq\max_{S\in \binom{V(H)}{t}}\sum_{(s,\vn,\ell)\in\mathcal{I}}d_{H_{s,\vn,\ell}}(S)\leq D^{s-t-\beta}.
\]

Finally, applying Theorem~\ref{thm:SmallCodegree} on $J$ and $H$, we obtain that there exists an $H$-avoiding $A$-perfect matching of $J$, which completes the proof.
\end{proof}

\section{Concluding Remarks}

We note that if $F$ is connected, then Theorem~\ref{thm:base} holds with $n$ replaced by $(k-1)\Delta(G)$. This is a natural bound in the ``sparse'' regime (when $\Delta(G)$ is much smaller than $n$). Unfortunately our proof of Theorem~\ref{thm:main2} does not immediately carry over to provide a logarithmic improvement over this simple bound in the sparse regime. The reason is that, while $F$ is connected, the forbidden configurations (the edges of $H$) are simply subsets of $F$ and hence may induce a disconnected graph. So we are unsure what the correct bound is in the sparse regime. We think it would be interesting to determine these numbers more precisely.

As mentioned in the introduction, Glock, Joos, Kim, K\"uhn and Lichev~\cite{GJKKL22} independently developed a similar theory of avoiding forbidden submatchings as in the Forbidden Submatching Method of~\cite{DP22}. The proof of Glock et.~al.~proceeds via a random greedy process while the proof in~\cite{DP22} uses the nibble method. One of the main results of~\cite{GJKKL22} is a weaker form of Theorem~\ref{thm:SmallCodegree} which essentially has two limitations compared to Theorem~\ref{thm:SmallCodegree}. First it only holds for dense graphs (where $D \ge {\rm polylog} |V(J)|$); second, it only finds an almost perfect matching in a general hypergraph $J$. While not immediately obvious, the bipartite $A$-perfect formulation in Theorem~\ref{thm:SmallCodegree} actually implies the coloring and list coloring versions of their results and more. Intriguingly, even though the main results of this paper, Theorems~\ref{thm:main1} and~\ref{thm:main2} are in the dense regime, the theorems do not follow from the results of Glock et.~al.~since the proof requires the $A$-perfect formulation of Theorem~\ref{thm:SmallCodegree} to actually find an $(F,q)$-coloring of all edges of $G$ instead of just for almost all edges of $G$.

We also note that the idea of excluding not just the natural forbidden configurations but rather a larger set including various other smaller forbidden configurations to aid with codegrees has been used before. For example, Glock, Joos, Kim, K\"uhn Lichev and Pikhurko~\cite{GJKKLP22} did this in their paper on the problem of Brown, Erd\H{o}s and S\'os. However, here we used a potential function to define our set of forbidden configurations; this notion of a potential for forbidden configurations was also implicitly used by the first author, Dudek and English~\cite{BDE22} (see their function `pow') but here the use is more clearly related to satisfying the codegree conditions of Theorem~\ref{thm:SmallCodegree}. Altogether, we believe such additional tricks for encoding problems in the framework of the Forbidden Submatching Method will prove useful for future applications.

In this paper, we also defined the generalized list Ramsey number. While we believe this notion is quite natural, we had not seen it in the literature to date. As seen in Theorem~\ref{thm:base}, the upper bounds from Theorems~\ref{thm:EG} and~\ref{thm:AFM} hold for list coloring. Of course, the lower bounds for coloring also carry over to list coloring. Thus the linear and quadratic thresholds for $r_{\ell}(K_n,K_p,q)$ are the same as for $r(K_n,K_p,q)$. Nevertheless, we wonder if the list version may be interesting in its own right. In particular, could there be better lower bound constructions when one is allowed to choose a list assignment? So we ask the following.

\begin{question}
Do $r(K_n,K_p,q)$ and $r_{\ell}(K_n,K_p,q)$ differ in order of magnitude for some $p$ and $q$? 
\end{question}

\begin{question}
More generally for a fixed positive integer $k \geq 2$, do $r(K_n^{(k)},K_p^{(k)},q)$ and $r_{\ell}(K_n^{(k)},K_p^{(k)},q)$ differ in order of magnitude for some $p$ and $q$? 
\end{question}

\bibliographystyle{plain}

\end{document}